\documentclass[english]{smfart}

\usepackage{amssymb,amsmath,amsrefs,mathrsfs,graphics,graphicx,float,mathptm}
\usepackage[T1]{fontenc}
\usepackage[english, francais]{babel}
\usepackage[all, knot]{xy}

\theoremstyle{definition}

\newtheorem*{ex*}{Example}
\newtheorem*{defn}{Definition}

\theoremstyle{plain}
\newtheorem{thm}{Theorem}[section]
\newtheorem{prop}[thm]{Proposition}
\newtheorem*{cor*}{Corollary}

\newtheorem{lem}[thm]{Lemma}
\newtheorem*{claim}{Claim}
\newtheorem{thm*}{Theorem}

\theoremstyle{remark}

\numberwithin{equation}{section}

\DeclareMathOperator{\codim}{codim}

\DeclareMathOperator{\End}{End}
\DeclareMathOperator{\Hom}{Hom}

\DeclareMathOperator{\Div}{div}

\DeclareMathOperator{\Span}{span}

\def\C{\mathbb{C}}

\def\P{\mathbb{P}}
\def\R{\mathbb{R}}
\def\Z{\mathbb{Z}}
\def\P{\mathbb{P}}

\def\O{{\mathscr{O}}}

\def\fan{\Delta}
\def\nn{{\bf n}}
\def\n0{{\bf n_0}}

\def\htop{h_{\rm top}}

\def\transp #1{\vphantom{#1}^{\mathrm t}\! {#1}}

\begin{document}

\title[Monomial Maps]
{Pulling Back Cohomology Classes and \\ Dynamical Degrees Of Monomial Maps}

\author{Jan-Li Lin}

\address{Department of Mathematics, Indiana University, Bloomington, IN 47405, USA}

\email{janlin@indiana.edu}



\begin{abstract}
We study the pullback maps on cohomology
groups for equivariant rational maps (i.e., monomial maps) on toric varieties. Our method
is based on the intersection theory on toric varieties. We use the method
to determine the dynamical degrees of monomial maps
and compute the degrees of the Cremona involution.
\end{abstract}

\maketitle


\section{Introduction}

For a meromorphic map $f:X\dashrightarrow X$
on a compact K\"{a}hler manifold $X$ of dimension $n$, and for $1\le k\le n$,
there is a well-defined pullback map $f^* : H^{k,k}(X;\R)\to H^{k,k}(X;\R)$. The {\em $k$-th dynamical degree}
of $f$ is then defined as
\[
\lambda_k(f) = \lim_{\ell\to\infty}  \bigl\| (f^\ell)^* \bigr\|^{1 / \ell}.
\]

The dynamical degrees
$\lambda_1(f),\cdots,\lambda_n(f)$ form an important family of birational invariants.
The computation and properties of the first dynamical degree of various maps is an active
research topic.
On the other hand, much less is known about the higher dynamical degrees.
In this paper, we focus on monomial maps, and we compute the higher dynamical degrees of dominant monomial maps.

Given an $n\times n$ integer matrix $\psi=(a_{i,j})$, the associated monomial map
$f:(\C^*)^n\to(\C^*)^n$ is defined by
\[
\textstyle{f(x_1,\cdots,x_n)=( \prod_{j=1}^n x_j^{a_{1,j}},\cdots ,\prod_{j=1}^n x_j^{a_{n,j}} ).  }
\]
The map $f$ give rise to an equivariant rational map on toric varieties which contain $(\C^*)^n$ as a dense open subset,
and $f$ is dominant if and only if $\det(\psi)\ne 0$.
The main theorem of this paper is the following.

\begin{thm*}
\label{thm:main}
Let $f$ be the monomial
map induced by $\psi$, and let $f$ be dominant. Let $|\mu_1|\ge |\mu_2|\ge \cdots\ge |\mu_n|$
be the eigenvalues (counting multiplicities) of $\psi$.
Then the $k$-th dynamical degree is
\[
\lambda_k(f) = |\mu_1\cdots\mu_k|.
\]
\end{thm*}

This answers Question 9.9 in \cite{HP}.
As a corollary, we also obtain the following formula for the topological entropy of monomial maps,
which answers Question 9.1 in \cite{HP}.

\begin{cor*}
The topological entropy of a dominant monomial map $f$ is
\[
\htop(f)= \sum_{|\mu_i|>1}\log|\mu_i|.
\]
\end{cor*}

The theory of toric varieties is a useful tool for understanding the dynamics of monomial maps,
as was shown in the papers (\cite{Fa, JW, ljl}). Using the intersection theory on toric varieties,
we are able to translate the computation of cohomology classes into much simpler computation
involving cones and lattice points.

We note that C. Favre and E. Wulcan \cite{FW} have obtained the same results on dynamical degrees using related but
different methods from ours.

In Section 2, we will review results from the intersection theory of toric varieties. Then
a method to pull back cohomology classes by monomial maps on toric varieties is developed in Section 3.
The method is applied in Section 4 to
give a concrete formula for pulling back cohomology classes by
monomial maps on $(\P^1)^n$. We use the formula to prove our
main result about dynamical degrees.
In the last section, we apply our method to compute the degrees of the Cremona involution $J$ on $\P^n$.

I am heartily thankful to my advisor Eric Bedford for his frequent conversation, encouragement and guidance.
I would also like to thank Julie D\'eserti and Tuyen Trung Truong for helpful discussions.


\section{Intersection theory on toric varieties}

In this section, we briefly review the intersection theory on toric varieties.
Results are stated without proof.
For more details, see \cite{FuSt}.

\subsection{Chow homology groups of toric varieties}

Following the notations of \cite{Fu, FuSt}, we let
$X=X(\fan)$ be a toric variety corresponding to a fan $\fan$ in a lattice $N$ of rank $n$.
Each cone $\tau\in\fan$ corresponds to an orbit $O_\tau$ of the torus action, and the
closure of $O_\tau$ in $X(\fan)$ is denoted by $V(\tau)$.
Each $V(\tau)$ is invariant under the torus action. Conversely,
every torus-invariant closed subvariety of $X$ is of the form $V(\tau)$.
The dimension of $V(\tau)$ is given by
\[
\dim(V(\tau))=n-\dim(\tau)=\codim(\tau).
\]
For each cone $\tau$, define $N_\tau$ to be the sublattice of $N$ generated by $\tau\cap N$,
and $M(\tau):=\tau^\bot\cap M$. Every nonzero $u\in M(\tau)$ determines a rational function
$\chi^u$ on $V(\tau)$. The divisor of $\chi^u$ is given by
\begin{equation}
\label{eq:chow_relation}
\Div(\chi^u)=\sum_\sigma \langle u, n_{\sigma,\tau} \rangle \cdot V(\sigma),
\end{equation}
where the sum is over all cones $\sigma\in\fan$ which contain $\tau$ with $\dim(\sigma)=\dim(\tau)+1$,
and $n_{\sigma,\tau}$ is a lattice point in $\sigma$ whose image generates the 1-dimensional lattice
$N_\sigma/N_\tau$.

For an arbitrary variety $X$, we define the {\em Chow group} $A_k(X)$ as the
group generated by the $k$-dimensional irreducible closed subvarieties of $X$, with
relations generated by divisors of nonzero rational function on some $(k+1)$-dimensional
subvariety of $X$. In the case of toric varieties, there is a nice presentation of
the Chow groups in terms of torus-invariant subvarieties and torus-invariant relations.

\begin{prop}
\mbox{}\label{prop:ch_homology}
\begin{itemize}
\item[(a)] The Chow group $A_k(X)$ of a toric variety $X$ is generated by the classes
$[V(\sigma)]$ where $\sigma$ runs over all cones of codimension $k$ in the fan $\fan$.
\item[(b)] The group of relations on these generators is generated by all relations (\ref{eq:chow_relation}),
where $\tau$ runs over cones of codimension $k+1$ in $\fan$, and $u$ runs over a
generating set of $M(\tau)$.
\end{itemize}
\end{prop}
For the proof, see \cite[Proposition 1.1]{FuSt}.

\subsection{Chow cohomology of toric varieties; Minkowski weights}

For a complete toric variety $X$, it is known (\cite[Proposition 1.4]{FuSt}) that
the Chow cohomology group $A^k(X)$ is canonically
isomorphic to $\Hom(A_k(X),\Z)$. Thus we can describe
$A^k(X)$ as follows.

Let $\fan$ be a complete fan in a lattice $N$, and let $\fan^{(k)}$ denote the subset
of all cones of codimension $k$.
\begin{defn}
An integer-valued function $c$ on $\fan^{(k)}$ is called a {\em Minkowski weight} of codimension $k$
if it satisfies the relations
\begin{equation}
\sum_{\sigma\in\fan^{(k)},~\sigma\supset\tau} \langle u, n_{\sigma,\tau} \rangle \cdot c(\sigma)=0
\end{equation}
for all cones $\tau\in\fan^{(k+1)}$ and all $u\in M(\tau)$.
\end{defn}

\begin{prop}
The Chow cohomology group $A^k(X)$ of a complete toric variety $X=X(\fan)$ is canonically isomorphic
to the group of Minkowski weights of codimension $k$ on $\fan$.
\end{prop}

This proposition is an immediate consequence of the description of $A_k(X)$ in Proposition
\ref{prop:ch_homology} and the isomorphism
$A^k(X)\cong\Hom(A_k(X),\Z)$.

\subsection{Relations with the usual (co)homology groups}

For a nonsingular complete toric variety $X=X(\fan)$, we have $H_k(X;\Z)=0$ and $H^k(X;\Z)=0$ for
$k$ odd (see \cite[p.92]{Fu}). Furthermore, $A_k(X)\cong H_{2k}(X;\Z)$ and $A^k(X)\cong H^{2k}(X;\Z)$
(see \cite[p.102 and p.106]{Fu}).
This means
\[
A_*(X)\cong H_{*}(X;\Z)\text{ and } A^*(X)\cong H^*(X;\Z),
\]
and both isomorphisms double the degree.

Moreover, suppose $X$ is also projective, thus K\"{a}hler. We have the Hodge decomposition
$H^k(X;\C)\cong \oplus_{p+q=k} H^{p,q}(X)$.
The following result is known for toric varieties (\cite[Proposition 12.11]{Dan}):
\[
H^{p,q}(X)\cong H^q(X,\Omega^p_X) = 0 \text{ for $p\ne q$.}
\]
As a consequence, we have $A^k(X)\otimes_\Z \R \cong H^{2k}(X;\R)\cong H^{k,k}(X;\R)$.


\section{Pulling back cohomology classes by a toric rational map}
\label{sec:pullback}

Let $X$ be a compact K\"{a}hler manifold of dimension $n$, and $f:X\dashrightarrow X$
be a dominant meromorphic map. For a smooth form $\theta$ on $X$, one defines the
{\em inverse image} of $\theta$ by $f$, denoted $f^*\theta$, as follows:
let $\Gamma_f\subset X\times X$ be the graph of $f$, and let $\tilde{\Gamma}_f$ be a
desingularization of $\Gamma_f$. One has the following commutative diagram
\[
\xymatrix{
& \tilde{\Gamma}_f\ar[ld]_{\pi_1}\ar[rd]^{\pi_2}\\
X\ar@{-->}[rr]_f & & X
}
\]
here $\pi_1,\pi_2$ are both holomorphic maps. We define
\[
f^*\theta = (\pi_1)_*(\pi_2^*\theta).
\]
This means, we interpret $\pi_2^*\theta$ as a current, and then push it forward by $(\pi_1)_*$.
The definition above does not depend on the choice of the desingularization
of $\Gamma_f$. Furthermore, $f^*$ induces a map on cohomology groups.

Although simple and intuitive, this definition is not easy to compute in general.
In the following, we provide a method to compute the pull back of cohomology classes
on toric varieties by toric maps.

\subsection{Constructing a good refinement $\tilde{\fan}$ according to $\fan$ and $\psi$}
For the homomorphism $\psi:N\to N$,
we want to obtain a refinement $\tilde{\fan}$ of the fan $\fan$ such that every cone in
$\tilde{\fan}$ is mapped by $\psi_\R := \psi\otimes_\Z \R$ into some cone in $\fan$.
Moreover, we want the fan $\tilde{\fan}$ to be nonsingular, and the corresponding
toric variety $X(\tilde{\fan})$ to be projective. The existence and construction of
such a refinement $\tilde{\fan}$ is shown in \cite[Section 1 and Section 2.1]{JW}.

Therefore, we obtain the following diagram:
\[
\xymatrix{
& X(\tilde{\fan})\ar[ld]_{\pi}\ar[rd]^{\tilde{f}}\\
X(\fan)\ar@{-->}[rr]_f & & X(\fan)
}
\]
where $\tilde{f}$ and $\pi$ are both toric morphisms.
The morphism $\tilde{f}$ is induced by $\psi$, and $\pi$ is a birational morphism
induced by the identity map on $N$.

\subsection{Pulling back by $\tilde{f}$}
The map $\tilde{f}: X(\tilde{\fan})\to X(\fan)$ is a toric morphism. This means that
if we let $N_\R:= N\otimes_\Z \R$, then the corresponding
linear transformation $\psi_\R: N_\R\to N_\R$ maps each cone $\tau'\in\tilde{\fan}$ into some cone
$\tau\in\fan$, i.e., $\psi_\R(\tau')\subset\tau$. Now, let $c\in A^k(X(\fan))$ be a Minkowski weight of
codimension $k$ on $\fan$. Since $\tilde{f}$ is dominant, we have the following formula for the
pull back $\tilde{f}^* c$ as a Minkowski weight of codimension $k$ on $\tilde{\fan}$
(see \cite[Proposition 2.7]{FuSt}). For $\tau'\in\tilde{\fan}^{(k)}$, let $\tau$ be the smallest
cone in $\fan$ that contains $\psi(\tau')$, then
\[
(\tilde{f}^*c) (\tau')= \begin{cases} [N:\psi(N)+N_\tau]\cdot c(\tau) &   \text{if $\codim(\tau)=k$,} \\
                  0 &   \text{if $\codim(\tau) < k$.}  \end{cases}
\]

\subsection{Push forward by $\pi$}
\label{sec:proj_formula}
In order to describe $\pi_*(\tilde{f}^* c)\in A^k(X(\fan))$, it is enough to know the values
$(\pi_*\tilde{f}^* c)(\sigma)$ for $\sigma\in\fan^{(k)}$.
The Poincar\'{e} duality holds for $X(\fan)$. This means, for each $\sigma\in\fan^{(k)}$, there is the Poincar\'{e}
dual of $[V(\sigma)]\in A_k$, denoted by $c_\sigma\in A^{n-k}$, such that
$c_\sigma(\tau) = [V(\sigma)] . [V(\tau)]$ for $\tau\in\fan^{(n-k)}$.
Let $deg:A_0(X(\fan))\to\Z$ be the {\em degree homomorphism}\footnote
{In \cite{FuSt}, the degree homomorphism is denoted by $deg$, so we use the same notion.
But we will use $\deg$ for another meaning. Hopefully it will be clear from the
context and cause no confusion.}
on the complete variety $X(\fan)$. Then we have
\[
\begin{split}
(\pi_*\tilde{f}^* c)(\sigma) &= deg \Bigl( (\pi_*\tilde{f}^* c) \cap [V(\sigma)] \Bigr) \\
                                 &= \Bigl( (\pi_*\tilde{f}^* c) \cup c_\sigma \Bigr) (\{0\}) \\
                                 &= \Bigl( (\tilde{f}^* c) \cup (\pi^* c_\sigma) \Bigr) (\{0\}).
\end{split}
\]
The first equality is a special case of \cite[Proposition 3.1(a)]{FuSt}. The second and the third equalities
are due to Poincar\'{e} duality and the projection formula
of intersection theory \cite[p.325]{FuInt}.

Finally, the cup product $(\tilde{f}^* c) \cup (\pi^* c_\sigma)$ applying on the cone $\{0\}$ can be
computed using \cite[Proposition 3.1(b)]{FuSt}:
\[
\Bigl( (\tilde{f}^* c) \cup (\pi^* c_\sigma) \Bigr) (\{0\})
= \sum_{(\sigma,\tau)\in \tilde{\fan}^{(k)}\times \tilde{\fan}^{(n-k)}} m_{\sigma,\tau}
  \cdot (\tilde{f}^* c)(\sigma) \cdot (\pi^* c_\sigma)(\tau).
\]
In the formula, we fix some generic vector $v\in N$, the sum is over all pairs
$(\sigma,\tau)$ such that $\sigma$ meets $\tau+v$, and $m_{\sigma,\tau}=[N:N_\sigma+N_\tau]$.
Here ``generic'' means outside a union of finitely many proper linear subspaces of $N_\R$.
For the proof of this theorem, and the precise specification of $v$ being generic, see \cite[Section 3]{FuSt}.


\section{Dynamical degrees of a monomial map}

The goal of this section is to compute the dynamical degrees of monomial maps.
Since the dynamical degrees are birational invariants (see \cite[Corollary 2.7]{Guedj}, we can choose any
birational model to compute them. For monomial maps,
the toric variety $(\P^1)^n$ is a convenient model. We start
by investigating the fan structure, the homology and the cohomology of $(\P^1)^n$,
then use these information to compute the dynamical degree.

\subsection{The fan structure of $(\P^1)^n$}
Throughout this section, let $\fan$ be the standard fan for $X(\fan)=(\P^1)^n$, which is as follows.
First, let $e_1,\cdots,e_n$ be the standard basis elements of $N_\R\cong\R^n$, then
the rays of $\fan$ are generated by $e_1,\cdots, e_n$ and $-e_1,\cdots, -e_n$.
In general, the $k$-dimensional cones are generated by $\{s_1 e_{i_1},\cdots,s_k e_{i_k}\}$, where
$i_1,\cdots,i_k$ are $k$ different numbers in $\nn =\{1,\cdots,n\}$, and $s_i\in\{-1,+1\}$.

Thus, $k$ dimensional cones are in one-to-one correspondence with pairs $(\alpha,s)$, where
$\alpha\subseteq\nn$ is a subset with $k$ elements, and $s$ is a map $s:\alpha\to \{-1,+1\}$.
We denote $s(i)$ by $s_i$.
Under this correspondence, we can denote a cone by $\sigma=(\alpha,s)$, and this means the cone
of dimension $|\alpha|$, generated by $\{s_i\cdot e_i\ |\ i\in \alpha\}$.

\subsection{The Chow homology group of $(\P^1)^n$}
We know from Proposition~\ref{prop:ch_homology} that $A_k((\P^1)^n)$ is generated
by the classes of subvarieties $V(\sigma)$
for $\sigma=(\alpha,s)$ with $|\alpha|=n-k$. Concretely,
in the coordinate $([x_1 : y_1],\cdots,[x_n : y_n])$ of $(\P^1)^n$, we have
\[
V(\sigma)= \{x_i=0 | i\in\alpha,\ s_i=+1\} \cap\{y_j=0 | j\in\alpha,\ s_j=-1\}.
\]

For two cones
$\sigma_1=(\alpha_1,s_1)$ and $\sigma_2=(\alpha_2,s_2)$ of the same dimension,
the relations between the cohomology classes are given by
\begin{equation}
\label{eq:rel_P1n}
[V(\sigma_1)]=[V(\sigma_2)] ~\Longleftrightarrow~ \alpha_1=\alpha_2 ~\Longleftrightarrow~
\Span(\sigma_1)=\Span(\sigma_2).
\end{equation}
Thus, if $s^+\equiv +1$ and let
$\sigma_\alpha = (\alpha, s^+)$, then the classes $[V(\sigma_\alpha)]$ will be a basis for $A_k((\P^1)^n)$.
For $\sigma_1=(\alpha_1,s_1)$ and $\sigma_2=(\alpha_2,s_2)$ of complementary dimensions,
i.e., $|\alpha_1|+|\alpha_2|=n$,
we have the following:
\[
[V(\sigma_1)] . [V(\sigma_2)]= \begin{cases} 1 &   \text{if $\alpha_1\cap \alpha_2=\emptyset$,} \\
                  0 &   \text{if $\alpha_1\cap \alpha_2\ne\emptyset$.}  \end{cases}
\]
Note that, since $|\alpha_1|+|\alpha_2|=n$, we have $\alpha_1\cap \alpha_2=\emptyset$ if and only if
$\alpha_1\cup \alpha_2=\nn$. We use the notation
$\alpha' = \nn-\alpha$ for the complement of $\alpha$. Thus the classes $[V(\sigma_{\alpha'})]$, with $|\alpha|=n-k$,
form a basis of $A_{n-k}((\P^1)^n)$, which is dual to the basis $[V(\sigma_{\alpha})]$ of
$A_k((\P^1)^n)$ under the intersection pairing.

\subsection{The Chow cohomology groups of $(\P^1)^n$}
By (\ref{eq:rel_P1n}),
we conclude that a function $c: \fan^{(k)}\to\Z$ is a Minkowski weight if we have
$c(\sigma_1)=c(\sigma_2)$ whenever $\Span(\sigma_1)=\Span(\sigma_2)$. Since
the Chow cohomology $A^k((\P^1)^n)$ is isomorphic to the group of Minkowski weights $\fan^{(k)}\to\Z$,
we can write down a basis for $A^k((\P^1)^n)$ as follows.
For each subset $\alpha\subseteq\nn$ with $|\alpha|=n-k$, denote
$E_\alpha=\Span\{e_i | i\in \alpha\}$, and
define $c_\alpha: \fan^{(k)}\to\Z$ as
\[
c_\alpha (\sigma)= \begin{cases} 1 &   \text{if $\sigma\subset E_\alpha$,} \\
                  0 &   \text{if $\sigma\not\subset E_\alpha$.}  \end{cases}
\]
These $c_\alpha$ form a basis for $A^k((\P^1)^n)$, and $c_\alpha$ is the
Poincar\'{e} dual to $[V(\sigma_{\alpha'})]$.

\subsection{Pulling back cohomology classes by monomial maps on $(\P^1)^n$}
Recall that $\psi\in M_n(\Z)$ is an integer matrix with nonzero determinant.
For any subsets
$\alpha, \beta \subseteq \nn$, we let
$\psi_{\alpha,\beta}$ denote the submatrix of $\psi$ obtained by the rows in
$\alpha$ and the columns in $\beta$. This means, if we write $\psi=(a_{i,j})$, then
$\psi_{\alpha,\beta}$ is formed by those $a_{i,j}$ with $i\in\alpha$ and $j\in\beta$.

\begin{thm}
\label{thm:pullback_on_P1n}
Let $f$ be the dominant monomial map of $(\P^1)^n$ induced by $\psi$. Then
the map $f^* : A^k((\P^1)^n)\to A^k((\P^1)^n)$
is represented, with respect to the basis $\{c_\alpha | \alpha\subset\nn,|\alpha|=n-k\}$
for $A^k((\P^1)^n)$, by the matrix whose $(\alpha,\beta)$-entry is
\[
\bigl|\det(\psi_{\beta',\alpha'})\bigr|.
\]
\end{thm}

Note that with $\psi:\Z^n\to\Z^n$, one define the exterior power
\[
\wedge^{k} \psi: \wedge^{k} \Z^n \to \wedge^{k} \Z^n
\]
by sending $v_1\wedge\cdots\wedge v_{k}$ to $(\psi v_1)\wedge\cdots\wedge (\psi v_{k})$.
For $\alpha = \{\alpha_1 < \alpha_2 < \cdots < \alpha_{n-k} \} \subseteq \nn$, we define
$e_\alpha = e_{\alpha_1}\wedge\cdots\wedge e_{\alpha_{n-k}}$.
Then $\{ e_\alpha  |  \alpha\subseteq\nn, |\alpha|=n-k \}$ form a basis for $\wedge^{n-k} \Z^n$,
and $\{ e_{\alpha'}  |  \alpha\subseteq\nn, |\alpha|=n-k \}$ form a basis for $\wedge^{k} \Z^n$.
They are dual basis in the sense that $e_\alpha\wedge e_{\alpha'}=\pm 1$ and $e_\alpha\wedge e_{\beta'}=0$
for $\alpha\ne \beta$.
With the basis $e_{\alpha'}$, the linear map $\wedge^{k} \psi$ is represented by the matrix
$\bigl(\det(\psi_{\alpha',\beta'})\bigr)$, i.e.,
\[
(\wedge^{k}\psi)(e_{\beta'}) = \sum_{\alpha'\subset\nn,|\alpha'|=k} \det(\psi_{\alpha',\beta'})\cdot e_{\alpha'}.
\]
Thus the matrix in the theorem is obtained by
taking absolute value of every
entry of the matrix representing the map
$\transp{(\wedge^{k} \psi)}$ on $(\wedge^{k} \Z^n)^*\cong \wedge^{n-k} \Z^n$.

We will need the following lemma.

\begin{lem}
\label{lem:first}
With the notations above, we have
\[
\bigl|\det(\psi)\cdot\det((\psi^{-1})_{\alpha,\beta})\bigr|=\bigl|\det(\psi_{\beta',\alpha'})\bigr|.
\]
\end{lem}

\begin{proof}
We first identify $\wedge^n\R^n\cong\R$ by sending $(e_1\wedge\cdots\wedge e_n)\mapsto 1$.
Then we have
\[
\det((\psi^{-1})_{\alpha,\beta}) = \pm \bigl((\wedge^{n-k}\psi^{-1})e_\beta\bigr)\wedge e_{\alpha'}.
\]
Multiplying by $\det(\psi)$ corresponds to applying $(\wedge^n\psi)$ on $\wedge^n\R^n$. Therefore,
\[
\begin{split}
\det(\psi)\cdot\det((\psi^{-1})_{\alpha,\beta})
&= \pm (\wedge^n\psi) \Bigl(\bigl((\wedge^{n-k}\psi^{-1})e_\beta\bigr)\wedge e_{\alpha'}\Bigr)\\
&= \pm e_\beta \wedge \bigl((\wedge^{k}\psi) e_{\alpha'}\bigr) \\
&= \pm \det(\psi_{\beta',\alpha'}).
\end{split}
\]
The lemma follows after we apply absolute value on both sides of the equation.
\end{proof}

\begin{proof}[Proof of Theorem \ref{thm:pullback_on_P1n}]
Throughout the proof, we let $\alpha,\beta$ be two proper subsets of $\n0$ with $|\alpha|=|\beta|=n-k$.
First, we have a projective, smooth resolution $X(\tilde{\fan})$
of $(\P^1)^n = X(\fan)$ with respect to the map $f$.
Recall that $f^* = \pi_* \circ \tilde{f}^*$.
If we write $N_\beta:=E_\beta\cap N$, then the pullback $\tilde{f}^*c_\beta$ is given by
\[
(\tilde{f}^*c_\beta) (\sigma)= \begin{cases} [N:\psi(N)+N_\beta]
                    &   \text{if $\psi(\sigma)\subset E_\beta$,} \\
                  0 &   \text{if $\psi(\sigma)\not\subset E_\beta$.}  \end{cases}
\]

Next, in order to find $\pi_* (\tilde{f}^*c_\beta)$,
we apply the projection formula method in section~\ref{sec:proj_formula} to $\sigma_\alpha\in\fan^{(k)}$:

\[
\begin{split}
(\pi_* (\tilde{f}^*c_\beta)) (\sigma_\alpha) &= (\tilde{f}^*c_\beta \cup \pi^* c_{\alpha'})(\{0\}) \\
&= \sum_{(\sigma,\tau)\in \tilde{\fan}^{(k)}\times \tilde{\fan}^{(n-k)}} m_{\sigma,\tau}
  \cdot (\tilde{f}^* c_\beta)(\sigma) \cdot (\pi^* c_{\alpha'})(\tau).
\end{split}
\]
Here the pullback $(\pi^* c_{\alpha'})$ is the Minkowski weight defined by
\[
(\pi^* c_{\alpha'}) (\tau)= \begin{cases} 1
                    &   \text{if $\tau\subset E_{\alpha'}$,} \\
                  0 &   \text{if $\tau\not\subset E_{\alpha'}$.}  \end{cases}
\]

Notice that $(\tilde{f}^* c_\beta)(\sigma)$ and $(\pi^* c_{\alpha'})(\tau)$ are both nonzero
if and only if $\sigma\subset \psi^{-1}(E_\beta)$ and $\tau\subset E_{\alpha'}$.
This means that we only need to look at the cones $\sigma$ and $\tau$
that are contained in the linear subspaces
$\psi^{-1}(E_\beta)$ and $E_{\alpha'}$, respectively.

Moreover, the sum is over all pairs
$(\sigma,\tau)$ such that $\sigma$ meets $\tau+v$ for a fixed generic $v\in N$. We know that, for a
generic $v$, the linear spaces $\psi^{-1}(E_\beta)$ and $E_{\alpha'}+v$
will either be disjoint or intersecting at one point. If they are disjoint, then the intersection
$(\psi^{-1}E_\beta)\cap E_{\alpha'}$ is of dimension $\ge 1$, and the vectors in the set
$\{~\psi^{-1}(e_i) | i\in\beta~\}\cup\{~e_j | j\in\alpha'~\}$ are linearly dependent.
This implies that the value $\det((\psi^{-1})_{\alpha,\beta})=0$,
hence $\bigl|\det(\psi_{\beta',\alpha'})\bigr| = \bigl|\det(\psi)\cdot\det((\psi^{-1})_{\alpha,\beta})\bigr| = 0$
by Lemma~\ref{lem:first}.
On the other hand, the sum is over the empty set, hence is also zero.

In the case that $\psi^{-1}(E_\beta)$ and $E_{\alpha'}+v$ intersect at one point for a fixed generic $v$, the
intersection point will be in the relative interior of cones,
and the sum contains exactly one term. In this case, we obtain the following:
\[
(\pi_* (\tilde{f}^*c_\beta)) (\sigma_\alpha) = [N: (\psi^{-1}E_\beta\cap N) + N_{\alpha'}] \cdot [N:\psi(N)+N_\beta].
\]
Note that the index $[N:\psi(N)+N_\beta]$ is always finite since we assume $\psi\otimes\R$ has full rank.
In the case that $\psi^{-1}(E_\beta)$ and $E_{\alpha'}+v$ intersect at one point, the index
$[N: (\psi^{-1}E_\beta\cap N) + N_{\alpha'}]$ is also finite.

Observe that the $(\alpha,\beta)$-entry of the matrix representing $f^*$ is exactly the number
$(\pi_* (\tilde{f}^*c_\beta)) (\sigma_\alpha)$.
By Lemma~\ref{lem:first}, the theorem will be true once we can prove the following:
\begin{lem}
\label{lem:second}
$[N: (\psi^{-1}E_\beta\cap N) + N_{\alpha'}] \cdot [N:\psi(N)+N_\beta]=|\det(\psi)\cdot\det((\psi^{-1})_{\alpha,\beta})|$.
\end{lem}

\begin{proof}[Proof of the lemma]
Without loss of generality, we assume that $\alpha=\beta=\{1,\cdots,n-k\}$.
First, observe that
\[
\begin{split}
|\det(\psi)| & = [N:\psi(N)] \\
             & = [N:\psi(N)+N_\beta] \cdot [\psi(N)+N_\beta:\psi(N)] \\
             & = [N:\psi(N)+N_\beta] \cdot [N_\beta:\psi(N)\cap N_\beta].
\end{split}
\]
The last equality is because $(\psi(N)+N_\beta) / \psi(N)$ is isomorphic to $N_\beta / (\psi(N)\cap N_\beta)$.

We observe that $\psi$ maps $(\psi^{-1}E_\beta)\cap N$ bijectively onto
$\psi(N)\cap N_\beta$, hence the two free abelian groups are isomorphic.
In particular, if $v_i = \transp{(a_{1,i},\cdots,a_{n-k,i},0,\cdots,0)}$, $i=1,\cdots,n-k$, is a basis of the
free abelian group $\psi(N)\cap N_\beta$, then $w_i=\psi^{-1}(v_i)$, $i=1,\cdots, n-k$
is a basis of $\psi^{-1}E_\beta\cap N$. If $w_i = \transp{(b_{1,i},\cdots,b_{n,i})}$,
we have the following matrix equation:
\begin{equation}
\label{eq:matrix_prod}
( w_1,\cdots,  w_{n-k}) = \psi^{-1}\cdot ( v_1,\cdots,  v_{n-k}).
\end{equation}
Here $( w_1,\cdots,  w_{n-k})$ and $( v_1,\cdots,  v_{n-k})$ denote the $n\times(n-k)$ matrices whose columns are formed
by the column vectors $w_1,\cdots,  w_{n-k}$ and $v_1,\cdots,  v_{n-k}$, respectively.
The bottom $k$ rows of the matrix $( v_1,\cdots,  v_{n-k})$ are all zero, and we
denote $\hat{v}_i = \transp{(a_{1,i},\cdots,a_{n-k,i})}$. Let $(\psi^{-1})_{\cdot,\beta}$ be the matrix formed by
the columns in $\beta$, i.e., the first $n-k$ columns under our assumption that $\beta=\{1,\cdots,n-k\}$.
Thus the right hand side of (\ref{eq:matrix_prod}) is determined by the product
$(\psi^{-1})_{\cdot,\beta} \cdot ( \hat{v}_1,\cdots,  \hat{v}_{n-k})$, and
we can reduce the matrix equation (\ref{eq:matrix_prod}) to
\begin{equation}
\label{eq:matrices}
( w_1,\cdots,  w_{n-k}) = (\psi^{-1})_{\cdot,\beta} \cdot ( \hat{v}_1,\cdots,  \hat{v}_{n-k}).
\end{equation}
Moreover, let $\hat{w}_i = \transp{(b_{1,i},\cdots,b_{n-k,i})}$, then we have an equation of submatrices
of (\ref{eq:matrices}):
\begin{equation}
\label{eq:sub_matrices}
( \hat{w}_1,\cdots,  \hat{w}_{n-k}) = (\psi^{-1})_{\alpha,\beta} \cdot ( \hat{v}_1,\cdots,  \hat{v}_{n-k}).
\end{equation}

The index $[N_\beta:\psi(N)\cap N_\beta]$
is equal to the
$(n-k)$-dimensional volume of the parallelogram formed by a basis of the free abelian subgroup
$(\psi(N)\cap N_\beta)\subseteq N_\beta$. By our assumption, $v_1,\cdots, v_{n-k}$ is such a basis, and
its $(n-k)$-dimensional volume is $|\det( \hat{v}_1,\cdots,  \hat{v}_{n-k}) |$.

For the index $[N: (\psi^{-1}E_\beta\cap N) + N_{\alpha'}]$,
a basis of $(\psi^{-1}E_\beta\cap N) + N_{\alpha'}$ is given by
\[
\{~w_1, \cdots, w_{n-k}, e_{n-k+1},\cdots,e_n ~\}.
\]
The volume of the associated parallelogram is then $|\det( \hat{w}_1,\cdots,  \hat{w}_{n-k} )|$.

By taking the determinant and absolute values on both sides of the equation (\ref{eq:sub_matrices}), we get
\[
[N: (\psi^{-1}E_\beta\cap N) + N_{\alpha'}] = |\det((\psi^{-1})_{\alpha,\beta})|\cdot [N_\beta:\psi(N)\cap N_\beta].
\]
We conclude that
\[
\begin{split}
 & [N: (\psi^{-1}E_\beta\cap N) + N_{\alpha'}] \cdot [N:\psi(N)+N_\beta] \\
=& [N: (\psi^{-1}E_\beta\cap N) + N_{\alpha'}] \cdot  \frac{|\det(\psi)|}{[N_\beta:\psi(N)\cap N_\beta]}\\
=& |\det((\psi^{-1})_{\alpha,\beta})|\cdot |\det(\psi)|.
\end{split}
\]
This proves the lemma.
\end{proof}
Therefore, by Lemma~\ref{lem:first} and Lemma~\ref{lem:second},
the proof of Theorem~\ref{thm:pullback_on_P1n} is complete.
\end{proof}

\subsection{Dynamical degrees of monomial maps}
Now we are ready to prove Theorem~\ref{thm:main}.

\begin{proof}[Proof of Theorem 1]
We will use $(\P^1)^n$ as our birational model to compute the dynamical degrees.
The iterate $f^\ell$ is induced by $\psi^\ell$.
With respect to the basis $\{c_\alpha | \alpha\subseteq\nn, |\alpha|=n-k\}$ for $A^k((\P^1)^n)$,
the space $\End(A^k((\P^1)^n))$ is isomorphic to $M_p(\Z)$ for $p$ equal to the binomial coefficient
$C(n,k)$. Let $\|\cdot\|$ be the norm defined by taking the maximum absolute value of coefficients. We have
\[
\begin{split}
   \lambda_k(f) & = \lim_{\ell\to\infty} \|(f^\ell)^*\|^{1/\ell} \\
    & = \lim_{\ell\to\infty} \Bigl\{\max_{\alpha,\beta}~\bigl|\det((\psi^\ell)_{\beta',\alpha'})\bigr|\Bigr\}^{1/\ell}\\
    & = \lim_{\ell\to\infty} \Bigl\| \wedge^{k} (\psi^{\ell})\Bigr\|^{1/\ell} \\
    & = \lim_{\ell\to\infty} \Bigl\|(\wedge^{k} \psi)^{\ell}\Bigr\|^{1/\ell} \\
    & = |\mu_1\cdots\mu_k|.
\end{split}
\]
This completes the proof.
\end{proof}

\begin{cor*}
The topological entropy of a birational monomial map $f$ is given by the following formula
\[
\htop(f)= \sum_{|\mu_i|>1}\log|\mu_i|.
\]
\end{cor*}

\begin{proof}
The ``$\ge$'' part is due to Hasselblatt and Propp (see \cite[Theorem 5.1]{HP}). For the ``$\le$'' part,
by \cite[Theorem 1]{DS},
\[
\htop(f)\le \max_{1\le k\le n}\log\lambda_k(f).
\]
For monomial maps, the number $\sum_{|\mu_i|>1}\log|\mu_i|$ is exactly the maximum
on the right side of the inequality. Thus the corollary is proved.
\end{proof}

Next, we consider monomial maps on $\P^n$.
Since $A^k(\P^n)\cong H^{2k}(\P^n;\Z)\cong\Z$ for all $k$,
the pull back $f^*$ of any rational map $f$ is given by an integer, called the $k$-th degree of $f$,
and is denoted by $\deg_k(f)$ (see Russakovskii and Shiffman \cite[p.914]{RS}).

Recall that the norm of $(f^\ell)^* := (f^\ell)^*|_{A^k((\P^1)^n)}$ is
\[
\|(f^\ell)^*\|=\|(\wedge^{k} \psi)^{\ell}\|,
\]
and the number $|\mu_1\cdots\mu_k|$ is the spectral radius of the map $(\wedge^{k} \psi)$.
By a standard result of linear algebra (see \cite[Lemma 6.4]{ljl}), we know that there exist two positive constants
$c_1\ge c_0 > 0$, and an integer $m$ with $0\le m\le C(n,k)$, such that
\[
c_0\cdot \ell^m \cdot |\mu_1\cdots\mu_k|^\ell \le \|(f^\ell)^*\| \le c_1\cdot \ell^m \cdot |\mu_1\cdots\mu_k|^\ell.
\]

Let $g: (\P^1)^n \to \P^n$ and $g^{-1}:\P^n\to(\P^1)^n$ be the birational map induced by identity map on the lattice $N$.
Then $f$ is birationally conjugate to the map $\tilde{f} := g\circ f \circ g^{-1}$ on $\P^n$.
The pullback map $(\tilde{f}^\ell)^*$ on $H^{k,k}(\P^n;\R)$ is represented by
the $k$-th degree  of $\tilde{f}^\ell$.
By a standard estimate on birational conjugation (see \cite[Corollary 2.7]{Guedj}),
we know that
\[
c'_0\cdot \|(f^\ell)^*\| \le \deg_k(\tilde{f}^\ell) \le c'_1\cdot \|(f^\ell)^*\|
\]
for some constants $c'_1\ge c'_0 > 0$. Combining the two estimates, we conclude the following estimate on the
growth of $k$-th degree on $\P^n$.

\begin{prop}
For a dominant monomial map $\tilde{f}$ on $\P^n$, there are two constants $C_1\ge C_0 > 0$
and an integer $m$ with $0\le m\le C(n,k)$, such that
the $k$-th degree of iterates of $\tilde{f}$ satisfies the following estimate:
\[
C_0\cdot \ell^m \cdot |\mu_1\cdots\mu_k|^\ell \le \deg_k(\tilde{f}^\ell) \le C_1\cdot \ell^m \cdot |\mu_1\cdots\mu_k|^\ell.
\]
\end{prop}
\hfill\qed


\section{The degrees of the Cremona involution on $\P^n$}
An important example of monomial maps is the Cremona involution $J:\P^n\dashrightarrow\P^n$.
It is the map defined by $J(z_1,\cdots,z_n)=(z_1^{-1},\cdots,z_n^{-1})$ on $(\C^*)^n$,
compactified to a birational map on $\P^n$. The main result of this section is the following

\begin{thm}
\label{thm:deg_J}
The $k$-th degree of the map $J$ is given by
\[
\deg_k(J)=C(n,k)=\frac{n!}{k!(n-k)!}
\]
for $k=0,\cdots,n$.
\end{thm}

The rest of this section is devoted to the proof of this theorem.
Throughout this section, $\fan$ will denote the standard fan associated to $\P^n$,
and $\tilde{\fan}$ will be the required refinement of $\fan$ with respect to the map $J$.

\subsection{Fan structure of $\fan$ and (co)homology of $\P^n$} Let $\n0 = \{ 0,1,\cdots,n\}$,
$e_1,\cdots,e_n$ be the standard basis of $N\cong\Z^n$, and $e_0=-(e_1+\cdots+e_n)$.
The cones in $\fan$ are generated by proper subsets of $\{e_0,e_1,\cdots,e_n\}$.
For any proper subset
$\alpha\subsetneq \n0$, we denote $\sigma_\alpha$ to be the cone generated by those $e_i$ where
$i\in\alpha$, i.e., $\sigma_\alpha=\{\sum_{i\in\alpha} a_i e_i  |  a_i\ge 0\}$.

The dimension of $\sigma_\alpha$ is $|\alpha|$, thus
$\sigma_\alpha\in\fan^{(k)} \Longleftrightarrow |\alpha|=n-k$, and $A_k(\P^n)$ is generated
by the classes of $V(\sigma_\alpha)$ for $|\alpha|=n-k$. The relations are given by
$[V(\sigma_{\alpha_1})]=[V(\sigma_{\alpha_2})]$ for any two $\alpha_1,\alpha_2\subset\n0$ with
$n-k$ elements.

As a consequence, a map $c: \fan^{(k)} \to \Z$ is a Minkowski weight if and only if it is a
constant function, i.e., $c(\sigma_1)=c(\sigma_2)$ for all $\sigma_1,\sigma_2\in \fan^{(k)}$.
Moreover, let $c_k: \fan^{(k)} \to \Z$ be the map defined by $c_k(\sigma)=1$ for all $\sigma\in\fan^{(k)}$,
then the Chow cohomology group $A^k(\P^n)\cong\Z$ is generated by $c_k$.

\subsection{The involution $J$ and its degrees} The map $J$ is induced by the linear map
$-I$, where $I$ is the identity map on $N$. Thus, to obtain the refinement $\tilde{\fan}$,
we need to subdivide $\fan$ so that every
cone after subdivision is mapped by $-I$ into some cone of $\fan$, and we have the following
diagram:
\[
\xymatrix{
& X(\tilde{\fan})\ar[ld]_{\pi}\ar[rd]^{\tilde{J}}\\
\P^n=X(\fan)\ar@{-->}[rr]_J & & X(\fan)
}
\]
Both the maps $\pi$ and $\tilde{J}$ are toric morphisms. This means, for each $\sigma'\in\tilde{\fan}$,
there exist cones $\sigma,\tau\in\fan$ such that $\sigma'\subseteq\sigma$ and $-\sigma'\subseteq\tau$.
Notice that $-\sigma'\subseteq\tau$ is equivalent to $\sigma'\subseteq -\tau$.

\begin{proof}[Proof of Theorem \ref{thm:deg_J}]
Let $c_k\in A^k(\P^n)$ and $c_{n-k}\in A^{n-k}(\P^n)$ be the generators of the two Chow cohomology
groups of complementary dimensions. Then by the formula in Section~\ref{sec:proj_formula},
the $k$-th degree of the map $J$ is given by
\begin{equation}
\label{eq:pullback_J}
\begin{split}
\deg_k(J) &= \Bigl( (\tilde{J}^* c_k) \cup (\pi^* c_{n-k}) \Bigr) (\{0\})\\
          &= \sum_{(\sigma',\tau')\in \tilde{\fan}^{(k)}\times \tilde{\fan}^{(n-k)}} m_{\sigma',\tau'}
             \cdot (\tilde{J}^* c_k)(\sigma') \cdot (\pi^* c_{n-k})(\tau'),
\end{split}
\end{equation}
where the sum is over all pairs
$(\sigma',\tau')$ such that $\sigma'$ meets $\tau'+v$ for a fixed generic $v\in N$.
We analyze this sum in the following.

First, for $\tau'\in\tilde{\fan}^{(n-k)}$, let $\tau\in\fan$ be the smallest cone such that $\tau'\subseteq\tau$.
We have the following formula for $\pi^* c_{n-k}$.
\[
(\pi^*c_{n-k}) (\tau')= \begin{cases} 1 &   \text{if $\codim(\tau)=n-k$,} \\
                  0 &   \text{if $\codim(\tau) < n-k$.}  \end{cases}
\]
Similarly, for $\sigma'\in\tilde{\fan}^{(k)}$, let $\sigma\in\fan$ be the smallest cone containing
$-\sigma'$, then
\[
(\tilde{J}^*c_{n-k}) (\sigma')= \begin{cases} 1 &   \text{if $\codim(\sigma)=k$,} \\
                  0 &   \text{if $\codim(\sigma) < k$.}  \end{cases}
\]
Therefore, in the sum (\ref{eq:pullback_J}), in order to have a nonzero summand, we must have
$\tau'\subseteq\tau$ and $\sigma'\subseteq -\sigma$ for some $\tau,\sigma\in\fan$
of codimensions $(n-k)$ and $k$, respectively.

Next, notice that for a fixed generic $v\in N$, the two sets $-\sigma$ and $\tau+v$ will either
be disjoint or intersect at exactly one point. If they are disjoint, then for all those
$(\sigma',\tau')\in \tilde{\fan}^{(k)}\times \tilde{\fan}^{(n-k)}$ with
$\tau'\subseteq\tau$ and $\sigma'\subseteq -\sigma$, the corresponding terms does not appear in the sum.
If $(-\sigma)\cap(\tau+v)=\{P\}$ a single point, then for generic $v$, we will have
$P\in \sigma'$ and $P\in \tau'+v$ for a unique pair
$(\sigma',\tau')\in \tilde{\fan}^{(k)}\times \tilde{\fan}^{(n-k)}$ with
$\tau'\subseteq\tau$ and $\sigma'\subseteq -\sigma$.
Only for this pair $(\sigma',\tau')$, the corresponding term will appear in the summand, and
the corresponding term is
\[
\begin{split}
 & m_{\sigma',\tau'} \cdot (\tilde{J}^* c_k)(\sigma') \cdot (\pi^* c_{n-k})(\tau') \\
=& [N:N_{\sigma'}+N_{\tau'}]\cdot 1\cdot 1 \\
=& [N:N_{\sigma}+N_{\tau}].
\end{split}
\]
The last equality is because of $N_{\sigma'}=N_{\sigma}$ and $N_{\tau'}=N_{\tau}$.

Therefore, we need to find out those $(\sigma,\tau)\in \fan^{(k)}\times\fan^{(n-k)}$
such that $(-\sigma)\cap(\tau+v)$ is a single point for a given generic $v$.
Without loss of generality, we may assume that
$v= - (v_1,\cdots, v_n)$ where $v_1,\cdots,v_n$ are all positive integers.
Then, we have
\[
(-\sigma)\cap(\tau+v)\ne\emptyset ~~~\Longleftrightarrow~~~ -v=(v_1,\cdots, v_n)\in \sigma + \tau.
\]
Let $\sigma=\sigma_\alpha$ and $\tau=\tau_\beta$ for $\alpha,\beta\subset\n0$
with $|\alpha|=n-k$ and $|\beta|=k$. We first want to show the following.
\begin{claim}
If $(v_1,\cdots, v_n)\in \sigma + \tau$, then $0\not\in\alpha$ and $0\not\in\beta$.
\end{claim}
Suppose, on the contrary, that $0\in\alpha$ or $0\in\beta$, then there must be some $i\in\{1,\cdots,n\}$
such that $i\not\in\alpha\cup\beta$. Without loss of generality, assume $1\not\in\alpha\cup\beta$.
Then since $e_0=(-1,\cdots,-1)$, and every point in $\sigma+\tau$ can be written in the form
\[
r_0\cdot e_0 + \sum_{i=2}^n r_i\cdot e_i \text{\ \ \ for $r_i\ge 0$},
\]
thus the first coordinate of every point in $\sigma+\tau$ is non-positive.
However, the first coordinate of $-v$ is $v_1>0$, they cannot be equal.
So we conclude that $0\not\in\alpha$ and $0\not\in\beta$.

Now that we know $\alpha\cup\beta \subset \{ 1,\cdots,n\}$, our next claim is the following.
\begin{claim}
We have $\alpha\cup\beta = \{ 1,\cdots,n\}$.
Thus, $\alpha$ and $\beta$ form a partition of the set $\{ 1,\cdots,n\}$.
\end{claim}
Suppose, for example, that $1\not\in\alpha\cup\beta$,
then the first coordinate of every point in $\sigma+\tau$ is $0$.
But, again, we have $v_1>0$ which leads to a contradiction, so the claim is proved

Finally, for every partition $\alpha$ and $\beta$
of the set $\{ 1,\cdots,n\}$ with $|\alpha|=n-k$ and $|\beta|=k$, there is
a unique point in $(-\sigma_\alpha)\cap(\tau_\beta+v)$.
For example, if $\alpha=\{1,\cdots,k\}$ and $\beta=\{k+1,\cdots,n\}$, then the unique point is
$(-v_1,\cdots,-v_k,0,\cdots,0)$. Furthermore, for every such partition $\alpha,\beta$
of $\{ 1,\cdots,n\}$, we have
\[
[N:N_{\sigma_\alpha}+N_{\tau_\beta}]=[N:N]=1.
\]
There are $C(n,k)$ such partitions of $\{ 1,\cdots,n\}$, each will contribute 1
to the summation in (\ref{eq:pullback_J}).
Therefore, we have $\deg_k(J)=C(n,k)$ and this completes the proof.
\end{proof}



\begin{bibdiv}
\begin{biblist}

\bib{Dan}{article}{
title={The geometry of toric varieties},
author={Danilov, V. I.},
journal={Russ. Math. Surv.},
volume={33},
date={1978},
pages={97--154}
}

\bib{DS}{article}{
title={Une borne sup\'{e}rieure pour l'entropie topologique d'une application rationnelle},
author={Tien-Cuong Dinh},
author={Nessim Sibony},
journal={Ann. Math.},
volume={161},
date={2005},
pages={1637--1644}
eprint={arXiv:0303271 [math.DS]}
}

\bib{Fa}{article}{
title={Les applications monomiales en deux dimensions},
author={Favre, Charles},
journal={Michigan Math. J.},
volume={51},
number={3},
date={2003},
pages={467--475}
eprint={arXiv:0210025 [math.CV]}
}

\bib{FW}{article}{
title={Degree growth of monomial maps},
author={Favre, Charles},
author={Wulcan, Elizabeth},
status={Preprint},
}



\bib{Fu}{book}{
title={Introduction to Toric Varieties},
author={Fulton, William},
series={Annal of Math Studies}
volume={131}
date={1993},
publisher={Princeton University Press},
address={Princeton, NJ}
}



\bib{FuInt}{book}{
title={Intersection Theory},
author={Fulton, William},
date={1998},
publisher={Springer-Verlag},
}


\bib{FuSt}{article}{
title={Intersection Theory on Toric Varieties},
author={William Fulton},
author={Bernd Sturmfels},
journal={Topology},
volume={36},
date={1997},
pages={335--354}
eprint={arXiv:9403002 [math.AG]}
}

\bib{Guedj}{article}{
title={Propri\'{e}t\'{e}s ergodiques des applications rationnelles},
author={Vincent Guedj},
eprint={arXiv:0611302 [math.CV]}
}

\bib{HP}{article}{
title={Degree-growth of monomial maps},
author={Hasselblatt, Boris},
author={Propp, James},
journal={Ergodic Theory Dynam. Systems},
volume={27},
number={5},
date={2007},
pages={1375--1397}
eprint={arXiv:0604521 [math.DS]}
}


\bib{JW}{article}{
title={Stabilization of monomial maps},
author={Jonsson, Mattias},
author={Wulcan, Elizabeth},
eprint={arXiv:1001.3938 [math.DS]}
}

\bib{ljl}{article}{
title={Algebraic stability and degree growth of monomial maps and polynomial maps},
author={Jan-Li Lin},
eprint={arXiv:1007.0253 [math.DS]}
}

\bib{RS}{article}{
title={Value distribution for sequences of rational mappings and complex dynamics},
author={Alexander Russakovskii},
author={Bernard Shiffman},
journal={Indiana Univ. Math. J.},
volume={46},
number={3},
date={1997},
pages={897-932}
eprint={arXiv:9604204 [math.CV]}
}

\end{biblist}
\end{bibdiv}
\end{document}